\documentclass[a4paper,12pt]{article}

\usepackage{amsmath}
\usepackage{amssymb}
\usepackage{amsthm}
\usepackage{mathtools}
\usepackage{parskip}
\usepackage{graphicx}
\usepackage{psfrag}
\usepackage{subfig}

\usepackage{aligned-overset}
\usepackage{authblk}
\usepackage{color}
\usepackage{enumitem}  
\usepackage[hmargin=2.5cm,vmargin=2.5cm]{geometry}

\newtheorem{theo}{Theorem}
\numberwithin{theo}{section} 

\newtheorem{lem}[theo]{Lemma}

\newtheorem{cor}[theo]{Corollary}

\usepackage{xcolor}

\newcommand{\R}{\mathbb R}

\usepackage{bbm}

\begin{document}
\begin{center}
\section*{Reaction-diffusion transport into core-shell geometry:  Well-posedness and stability of stationary solutions}
\end{center}

\begin{center}
\large{Thomas Geert de Jong\\}
\footnotesize{Faculty of Mathematics and Physics, 
Institute of Science and Engineering, \\
Kanazawa University,
Kanazawa, Japan.\\
{\tt tgdejong@se.kanazawa-u.ac.jp}}
\\[2mm]

\large{Georg Prokert\\}
\footnotesize{Faculty of Mathematics and Computer Science,\\
Center for Analysis, Scientific Computing and Applications, \\
Eindhoven University of Technology, Eindhoven, The Netherlands. \\
{\tt g.prokert@tue.nl}}
\\[2mm]

\large{Alef Sterk\\}
\footnotesize{
Bernoulli Institute for Mathematics, Computer Science and Artificial Intelligence,\\
University of Groningen, Groningen, The Netherlands. \\
{\tt a.e.sterk@rug.nl}}
\end{center}

\begin{quote}
\noindent \textbf{Abstract:} We investigate a nonlinear parabolic reaction-diffusion equation describing the oxygen concentration in encapsulated pancreatic cells with a general core-shell geometry. This geometry introduces a discontinuous diffusion coefficient as the material properties of the core and shell differ. We apply monotone operator theory to show well-posedness of the problem in the strong form.  Furthermore, the stationary solutions are unique and asymptotically stable. These results rely on the gradient structure of the underlying PDE.

\noindent \textbf{Keywords:} parabolic PDE, reaction-diffusion, diabetes, pancreas, asymptotic stability, gradient flow, monotone operator theory.  \\
\end{quote}

\section{Introduction}

Type 1 diabetes is a chronic condition that is caused by the immune system mistakenly destroying insulin producing pancreatic Langerhans islets. Treatment requires continuous monitoring and maintaining of insulin levels via external means such as injections or insulin pump. Islet transplantation from deceased donors provides a new treatment to recover natural insulin production \cite{shapiro2017clinical}. This is not a permanent solution as a significant percentage of patients do not achieve insulin independence at 5 years \cite{shapiro2017clinical}. Furthermore, such a treatment requires immunosuppressants. As there are only few suitable donors it is desirable to increase the durability of the transplanted cells.  This could be achieved by protecting the transplanted cells with a physical barrier. Alginate is a promising candidate as it can selectively diffuse or block certain molecules \cite{GOO85}. In addition, alginates are relatively inert \cite{AUG06,espona2018tunable,espona20193d}. This novel treatment shows promising results in mammalian trials \cite{elliott2005intraperitoneal,dufrane2010alginate,alagpulinsa2019alginate, bochenek2018alginate}. Human trials are still in the preliminary phase \cite{basta2011long,calafiore2006microencapsulated,tuch2009safety}. 

King et al.~\cite{KING19,king2020corrigendum} proposed a model that describes the reaction-diffusion of oxygen through a protective shell encapsulating a core so that hypoxia can be avoided within the donor cells. The encapsulation and core of donor cells are approximately spherical \cite{MA13}. Hence, in \cite{KING19,king2020corrigendum} a spherical core and shell with common center are considered. In that paper the authors derived the governing ODE for stationary solutions and numerically computed these solutions with oxygen concentration above the hypoxia threshold.  The existence of these stationary solutions is made rigorous in \cite{JON2021topological} using topological shooting \cite{PEL95,JON20,HAS12}. The latter result considers parameters used in cell encapsulation experiments\cite{AVG07,BUCH09,MA13}. 

The oxygen concentration in the core and shell will initially be away from the stationary state. Furthermore, the geometry of the cell is only approximately spherical.  Hence, it is important for the validation of the model to determine the stability of stationary solutions for general core-shell geometry which is the goal of this work. This first requires formulating the corresponding PDE and showing its well-posedness which was not considered in \cite{KING19,king2020corrigendum}. The resulting PDE is of parabolic type. On the boundary of the outside shell we assume the oxygen concentration to be constant. As the shell and the core have different diffusion coefficients, there is a discontinuity of this coefficient at the interface separating them. This discontinuity makes the corresponding stationary problem a diffraction problem \cite{ladyzhenskaya2013boundary}. We make the natural assumption that the concentration and the flux of oxygen are continuous at the interface. Oxygen is consumed only by the donor cells in the core but not by the protective shell. This leads to a nonlinear reaction-diffusion PDE in the core, where the non-linearity corresponds to Michaelis-Menten consumption and consequently is bounded and monotone. In the shell, however, the problem reduces to linear diffusion.

Our main results for the governing PDE are
\begin{itemize}
\item[-] well-posedness of the strong form,
\item[-] uniqueness and asymptotic stability of stationary solutions. 
\end{itemize}
The results are necessary theoretical steps in the validation of the biological model.

The sum of the reaction term and the spatial differential term is a nonlinear monotone operator. Hence, for the well-posedness result we can apply classical monotone operator theory. We apply the well-known theorems of Komura and Browder-Minty \cite{zeidler2013nonlinear}.  We emphasize that the monotonicity of the nonlinear operator defining our evolution problem results from the underlying structure of the problem as a gradient flow with respect to a convex functional. This structure also enables us to show uniqueness and asymptotic stability of the stationary solution.

The governing equations are presented in Section \ref{sec:eq}. The well-posedness result is proved in Section \ref{sec:wp}. The stability results are presented in Section \ref{sec:stab}. Finally, in Section \ref{sec:conc} we conclude with remarks and an outlook on further questions.

\section{The model \label{sec:eq}}

We start with a description of the geometry of the core and its protective shell.  The domain is denoted by $\Omega \subset \mathbb{R}^N$ with $\overline{\Omega}$  compact and $N \geq 2$. For the application $N=3$ but the results hold for $N \geq 2$. Let $\Gamma$ be an $N-1$-dim $C^2$ surface that divides $\Omega$ into two open domains $\Omega_1, \Omega_2$ such that $\partial\Omega_1=\Gamma$, closed connected hypersurface, $\Omega_2$ bounded, with 2 boundary components, $S:=\partial \Omega$ and $\Gamma$, see Figure \ref{fig:cell}. We take $\Gamma,S$ piecewise $C^1$.  We refer to $\Omega_1$ as the core, $\Omega_2$ as the shell and $\Omega$ as the core-shell.

\begin{figure}[ht]
    \centering
    \includegraphics[width=7cm]{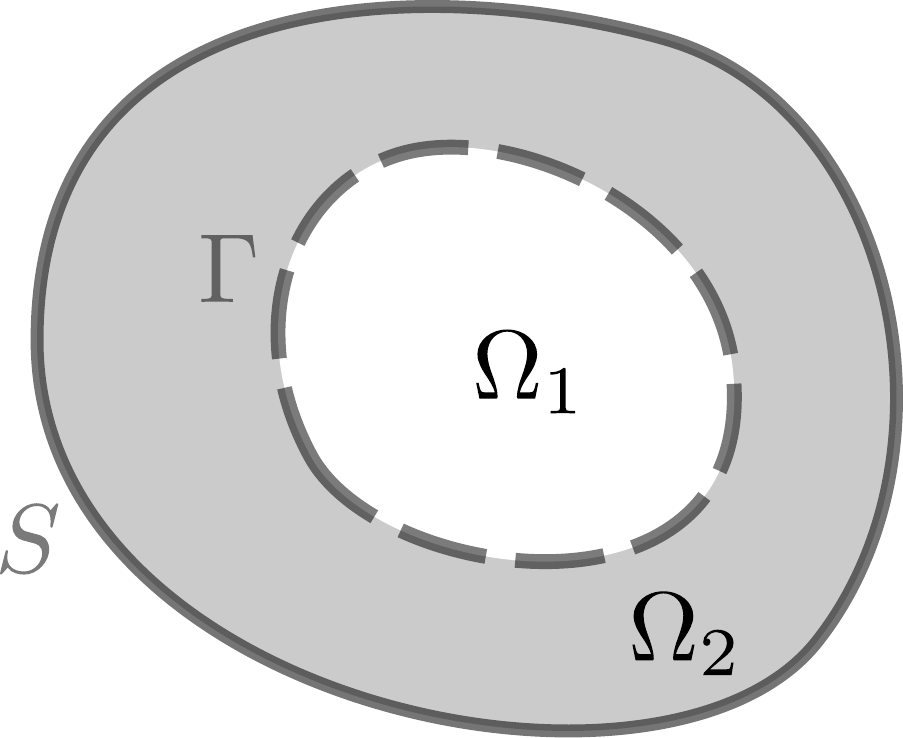}
    \caption{Core-shell geometry: $\Omega_1$ (white) corresponds to the core and $\Omega_2$ (grey) corresponds to the shell.\label{fig:cell}}
\end{figure}

On $S$ we consider a Dirichlet boundary condition, i.e. we assume that the transplanted cells are injected in an oxygen stable environment \cite{shapiro2017clinical}. At $\Gamma$ we require the concentration and flux to be continuous. Finally, on $\Omega_1$ there is a non-linear term corresponding to Michaelis-Menten consumption \cite{KING19}. It is assumed that on $\Omega_2$ there is no oxygen consumption. We also assume that Michaelis-Menten consumption is zero when the oxygen concentration is zero. This condition ensures that the non-linear term is bounded and monotone. 

The untransformed non-dimensional oxygen concentration equations are given in Appendix \ref{app:orig0}. Under a linear transformation the non-dimensional oxygen concentration equations become: 
\begin{align}
\frac{du}{dt} - b \Delta u &=  f(u) \; \; {\rm in} \; \; (\Omega \backslash \Gamma) \times (0,T), \label{eq:govp}\\
u |_S &= 0  \label{eq:boundp},\\
[u]_\Gamma &= 0 \label{eq:contup},\\
\left[ b \nabla u \cdot \nu \right]_{\Gamma} &=0, \label{eq:contfluxp} \\
u(\cdot,0) &= u_0 \; {\rm in} \; \Omega. \label{eq:ut0}
\end{align}
In Equation \eqref{eq:govp}  $b: \Omega \rightarrow \mathbb{R} $ is given by 
 \begin{align*}
b(x) = \begin{cases}
b_1 & {\rm if\; } x \in \overline{\Omega}_1, \\
b_2 & {\rm if\; }x  \in \Omega_2 ,
\end{cases}
\end{align*}
with $b_1,b_2 >0$, the map $f: L^2(\Omega)\rightarrow L^2(\Omega)$ given by 
\begin{align}
[f(u)](x)=\mathbbm{1}_{\Omega_1}(x)\varphi(u(x)),  \label{eq:newf}
\end{align}
with the $\varphi:\R \rightarrow\R$ given by
\begin{align}
\varphi(z)=\left\{\begin{array}{cl}
\frac{c_0-z}{c_1-z}&\text{ if $z\leq c_0$,}\\
0&\text{ if $z>c_0$,}
\end{array}\right. \label{eq:phi}
\end{align}
where $0 < c_0 < c_1$, see Figure \ref{fig:phi}.  \\

In \eqref{eq:contup}, \eqref{eq:contfluxp} the square brackets $[\cdot]_\Gamma$ are denoting the jump of  the quantity in the  brackets across 
$\Gamma$, i.e. the trace from $\Omega_2$ minus the trace of $\Omega_1$.  In \eqref{eq:contfluxp} $\nu$ is the normal directed towards $\Omega_2$. 

\begin{figure}[ht]
    \centering
    \includegraphics[width=8cm]{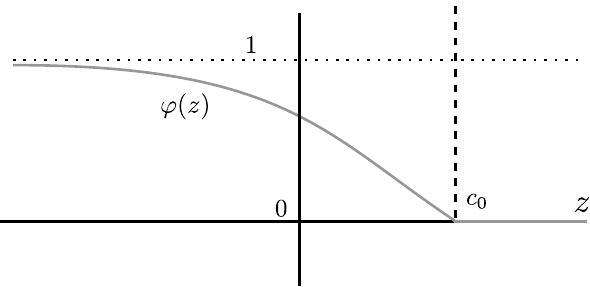}
    \caption{The function $\varphi$ is decreasing. This function models the oxygen consumption of the core cells. The value $\varphi=0$ for $z > c_0$ corresponds to zero oxygen consumption for zero oxygen concentration. In the application this means that the cells are dying from hypoxia. \label{fig:phi}}
\end{figure} 

It is straightforward to see following properties on $\varphi$: 

\begin{lem}
$\varphi$ defined in \eqref{eq:phi} satisfies 
\begin{itemize}
\item[-] $0\leq\varphi(z)<1$ for all $z\in\R$,
\item[-] $\varphi$ decreasing,
\item[-] $\varphi$ is continuous on $\R$ and smooth outside $c_0$,
\item[-] $\varphi$ is Lipschitz continuous with Lipschitz constant $1/(c_1-c_0)$, %$|\varphi'(z)|\leq 1/(c_1-c_0)$ for all $z \neq c_0$,
\item[-] $z\varphi(z)\leq c_0$ for all $z\in\R$. % (follows from $0\leq\varphi(z)<1$).
\end{itemize}
\label{lem:phiprop}
\end{lem}

\section{Well-posedness results \label{sec:wp}}

Define $V:= H^{1}_0(\Omega), H:=L^2(\Omega)$. So we have the evolution triplet 
\[
V \subset \subset H=H^\ast  \subset V^\ast,
\]
where $H^\ast$ and $V^\ast$ denote the dual space of $H$ and $V$, respectively. Moreover, we write $V\subset\subset H$ to emphasize the compactness of the embedding of $V$ in $H$. The inner product on $V$ is defined by $(u,v)_V = (u,v)_H + (\nabla u, \nabla v)_H$. Denote by  $\langle \cdot, \cdot \rangle$ the pairing between $V^{*}$ and $V$. 

Define the nonlinear operator $A : D(A) \rightarrow H$ given by
\begin{align}
A(u) = -b \Delta u - f(u), \label{eq:A}
\end{align}
with 
\[
D(A) := \{ u \in V \; : \; u|_{\Omega_i} \in H^2(\Omega_i), \; \; u \;\; {\rm satisfies} \;\; \eqref{eq:contfluxp} \}.
\]

We consider the equation
\begin{align}
\frac{du}{dt} +A(u) =0 , \label{eq:nonlingov}
\end{align}
as an equality in $L^2(0,T;H)$. 
% We apply a result of Komura \cite{komura1967nonlinear} as formulated by Zeidler:

\begin{theo}
For $u(0)=u_0 \in D(A)$ Equation \eqref{eq:nonlingov}  has a unique solution $u \in C^0(0,T ;  H)$ for any $T>0$  with
\[
u \in {\rm Lip}(0, T; D(A)) , \qquad \frac{du}{dt} \in C_w(0, T; H).
\]
Furthermore, $u_0 \mapsto u(t) \in C^0(D(A),D(A))$.  \label{theo:wp}
\end{theo}

\begin{proof} We will apply Theorem 31.A from \cite{zeidler2013nonlinear}. Recall that $A$ is called monotone if $( {A}(u)-{A}(v),u-v )_H \geq 0 $ for all $u,v \in D(A)$. The assumptions to check are:
\begin{itemize}
\item[(H1):] $A$ is monotone,
\item[(H2):] $R(I+A)=H$.
\end{itemize}
Proof (H1): Since $\varphi$ is decreasing we obtain
\begin{align}
( {A}(u)-{A}(v),u-v )_H = \int_\Omega b \nabla (u-v) \cdot \nabla (u-v) dx   - \int_{\Omega_1} \underbrace{( \varphi(u)- \varphi(v))(u-v)}_{\leq 0}  dx \geq 0 \label{eq:H1}
\end{align}
for $u,v \in D(A)$. 

Proof (H2): We consider $\tilde{A}: V \rightarrow V^\ast$  given by 

\[
\langle \tilde{A}(u),w \rangle := \int_\Omega \left( uw + b \nabla u \nabla w \right) dx - \int_{\Omega_1} \varphi(u) w dx\; \; u,w \in V.
\]

Recall that $\tilde{A}$ is called hemicontinuous if the real function $t \mapsto \langle \tilde{A}(u+tv),w \rangle$ is continuous on $[0,1]$  for all $u,v,w \in X$ and that $\tilde{A}$ is called coercive if 
\begin{align}
\frac{\langle \tilde{A}(u) , u \rangle}{\| u \|_V} \rightarrow + \infty  \;\; {\rm as \;\;} \|u \|_V \rightarrow \infty. \label{def:coer}
\end{align}
We apply Theorem 26.A  \cite{zeidler2013nonlinear} to $\tilde{A}$ which gives that for fixed $g \in H$ there exists an $u \in V$ such that $\tilde{A} u =g$. The assumptions to check are: 

\begin{itemize}
\item[-] $\tilde{A}$ monotone:  Take $u,v \in V$ then similarly to \eqref{eq:H1} we have that 
\[
\langle \tilde{A}(u)- \tilde{A}(v), u-v \rangle  \geq \int_\Omega b \nabla (u-v) \cdot \nabla (u-v) dx   - \int_{\Omega_1} ( \varphi(u)- \varphi(v))(u-v) dx \geq 0 .
\]
\item[-] $\tilde{A}$ coercive:  Recall from Lemma \ref{lem:phiprop} that  $\varphi(z)z \leq c_0$  for all $z \in \R$. Hence, we have that
\[
\langle \tilde{A}(u) , u \rangle = \int_\Omega  (u^2 + b | \nabla u |^2) dx - \int_{\Omega_1} \varphi(u) u dx \geq c \| u \|_V^2- |\Omega_1| c_0
\] 
and therefore we have \eqref{def:coer}.
%\[
%\frac{\langle \tilde{A}(u) , u \rangle}{\| u \|_V} \rightarrow + \infty  \;\; {\rm as \;\;} \|u \|_V \rightarrow \infty.
%\]
\item[-] $\tilde{A}$ hemicontinuous: for $t,s \in [0,1]$ $u,v,w \in V$, the Lipschitz continuity of $\varphi$ as established in Lemma \ref{lem:phiprop}, we have 
\begin{align*}
|\langle \tilde{A}(u + tv), w \rangle - \langle \tilde{A}(u + sv), w \rangle | & \leq  C \| v \|_V \|w\|_V |t-s| \\
& \qquad  + \left| \int_{\Omega_1} \left( \varphi(u+tv)- \varphi(u+sv) \right) w dx \right| \\
& \leq C \| v \|_V \|w\|_V |t-s| + L \| v \|_H \|w \|_H |t-s|
\end{align*}
This implies the continuity of $t \mapsto \langle \tilde{A}(u + tv), w \rangle$
\end{itemize}

Hence, for fixed $g \in H$ there exists an $u \in V$ such that $\tilde{A} u =g$, i.e.
\begin{align}
\int_\Omega \nabla u \nabla w dx = \int_\Omega (-u + f(u)+g)w dx \; \; \forall w \in V. \label{eq:tildeAuisg}
\end{align}
It remains to show that $u \in D(A)$. Since \eqref{eq:tildeAuisg}, $-u + f(u)+g \in H$, and  $b$ is constant on $\Omega_i$ we can adapt standard arguments on regularity of weak solutions to elliptic boundary value problems to obtain
\[
u|_{\Omega_i} \in H^2(\Omega_i) \; \;{\rm and} \;  \; b \Delta u = -u + f(u) + g \;\; {\rm a.e.} \; {\rm in} \; \Omega_i, \;i=1,2. 
\] 
Hence, applying integration by parts in the subdomains,
\begin{align*}
\int_\Omega b\nabla u\nabla w\,dx
&=\int_{\Omega_1} b\nabla u\nabla w\,dx+\int_{\Omega_2} b\nabla u\nabla w\,dx\\
&=-\int_{\Omega_1} b\Delta u w\,dx-\int_{\Omega_2} b\Delta u w\,dx-\int_\Gamma[b\nabla u\cdot\nu]_\Gamma w\,ds\\
&=\int_\Omega (-u+f(u)+g)w\,dx-\int_\Gamma[b\nabla u\cdot\nu]_\Gamma w\,ds\quad\forall w\in V,
\end{align*}
so by \eqref{eq:tildeAuisg}, the boundary integral vanishes for all $w\in V$ which implies \eqref{eq:contfluxp}, and $u\in D(A)$ is proved.\\

Finally, applying Corollary 31.1 from \cite{zeidler2013nonlinear} gives that $u_0 \mapsto u(t) \in C^0(D(A),D(A))$.
\end{proof}

\section{Stability results \label{sec:stab}}

We define the functional $E: V \rightarrow \R$ given by 
\begin{align}
E(u) = \int_\Omega \frac12 b |\nabla u|^2 dx  - \int_{\Omega_1} F(u) dx, \label{eq:J}
\end{align}
where 
\begin{align}
F(s) = \begin{cases} 
s + (c_1-c_0) \log{\frac{c_1 -s}{c_1}} & \; {\rm if} \quad s  \leq c_0, \\
c_0 + (c_1-c_0) \log{\frac{c_1 -c_0}{c_1}} & \; {\rm if} \quad  s  >  c_0.
\end{cases}
\end{align}

\begin{lem}
$E$ is Fr\'{e}chet differentiable with derivative at $u$ given by
\[
E'(u)[h] = \int_{\Omega} b \nabla u \nabla h dx - \int_{\Omega_1} \varphi(u)h dx.  
\]
\label{lem:Jcrit}
\end{lem}

\begin{proof} The first term in \eqref{eq:J} is a quadratic term on $V$.  $F$ has a bounded, integrable weak second derivative. Hence, for $z, \zeta \in \R$ we have 
\[
\left| F(z + \zeta) - F(z) - F'(z) \zeta \right| = \left|  \zeta^2 \int^{1}_0 (1-s)F''(z+s\zeta)ds \right| \leq \frac{1}{2} \| F''\|_\infty \zeta^2.
\]
Thus, for $u,h \in V$
\[
\left| \int_{\Omega_1} F(u+h) - F(u) - F'(u)h \right| \leq \frac{1}{2} \|F''\|_\infty \| h \|_V^2,
\]
and $u \mapsto \int_{\Omega_1}F(u)dx$ is Fr\'{e}chet differentiable with derivative given by 
\[
h \mapsto  \int_{\Omega_1} F'(u)h dx  =  \int_{\Omega_1} \varphi(u)h dx .
\]
\end{proof}

\begin{lem}  $E'$ is strongly monotone, i.e. for all $u,v \in V$ there exists $\gamma >0$ such that
\[
(E'(u)-E'(v))[u-v]  \geq \gamma \Vert u-v \Vert_V^2.
\]
\label{lem:strongmono}
\end{lem}
\begin{proof}  Take $u,v \in V$, then using $\varphi$ decreasing  and Poincar\'{e}'s inequality we obtain
\begin{align*}
(E'(u)-E'(v))[u-v] &= \int_{\Omega} b (\nabla u -  \nabla v) \cdot (\nabla u -  \nabla v) dx - \int_{\Omega_1} \underbrace{( \varphi(u)- \varphi(v))(u-v)}_{\leq 0} dx , \\
&  \geq \gamma \Vert u-v \Vert_V^2 .  
\end{align*}
\end{proof}

From Lemma \ref{lem:Jcrit} it follows that we can write \eqref{eq:nonlingov} as a gradient flow:
\begin{align*}
u_t = - \nabla E(u),
\end{align*}
in the sense that
\begin{align*}
(u_t,w)_H = -E'(u)[w]. 
\end{align*}
Recall that $u^\ast \in V$ is called a critical point if $E'(u^\ast)v =0$. Hence, stationary solutions are critical points of $E$.

\begin{theo} There is precisely one stationary solution to \eqref{eq:nonlingov}.
\label{theo:uniquestat}
\end{theo}

\begin{proof} If $E$ is a continuous, strictly convex, coercive functional then $E$ has  a unique critical point which is also a global minimum by Theorem 1.5.6 and 1.5.7 in \cite{badiale2010semilinear}.  $E$ is continuous. Strict convexity of $E$ follows from Lemma \ref{lem:strongmono}. So we only need to show coercivity. Observe that if $s > 0$ then $F(s) \leq s$ and if $s \leq 0$ then $F(s) \leq 0$. Hence, $F(s) \leq |s|$ for all $s \in \R$. Therefore, using Poincar\'{e}'s inequality we have that 
\[
E(u) \geq  C  \|u\|_V^2 -  C \| u \|_V.
\]
Hence, $E$ is coercive.   The result follows from Lemma \ref{lem:Jcrit}.
\end{proof}

Denote the unique stationary solution by $u^\ast \in V$.  
 
\begin{theo}[Global asymptotic stability in $H$] Let $u$ be a solution of \eqref{eq:nonlingov} then there exists a $\beta >0$ such that $\lim_{t \rightarrow \infty} {\rm e}^{\beta t} \|u(t)- u^\ast \|_H =0$.
\label{theo:asymp}
\end{theo}

\begin{proof}
 Observe that $E'(u^\ast)=0$. From $u \in L^2(0,T; V)$, $\frac{du}{dt} \in L^2(0,T; V^\ast)$ we have that
\begin{align}
\frac{d \|u \|_H^2}{dt} =2 \left(\frac{du}{dt},u \right)_H, \label{eq:uct}
\end{align}
by Theorem 7.2 p. 191 in \cite{robinson2001infinite}. Now let $u$ be the solution of \eqref{eq:nonlingov}. Then by \eqref{eq:uct} and Lemma \ref{lem:strongmono} we have that 
\begin{align*}
\frac12 \frac{d}{dt} \left(  \| u-u^\ast\|_H^2 \right) &= \left( \frac{du}{dt}, u-u^\ast \right)_H = -E'(u)[u-u^\ast], \\
&\leq   (E'(u^\ast)-E'(u))[u-u^\ast] \leq  - \gamma\| u-u^\ast\|_V^2 \leq - \gamma\| u-u^\ast\|_H^2. 
\end{align*}
\end{proof}

Asymptotic stability in the $V$-norm does not follow from Theorem \ref{theo:asymp} but we can obtain a weak $V$-stability result:

\begin{cor} $u(t) \rightharpoonup u^\ast$ in $V$ as $t \rightarrow \infty$.
\label{cor:Vweakstable}
\end{cor}

\begin{proof} We first show that 
\begin{align}
t \mapsto u(t) \; {\rm is \; bounded \; in \;} V. \label{uVbounded}
\end{align}
We can bound $\| u \|_V$ in terms of $E(u)$ and $\| u \|_H$:
\[
\| u \|_V^2 \leq C \int_{\Omega} b | \nabla u |^2 dx = 2C \left( E(u) + \int_{\Omega_1} F(u) dx \right) \leq C' (E(u) + \|u \|_H).
\]
Now \eqref{uVbounded} follows because $E(u)$ is decreasing and $\|u\|_H$ is bounded by Theorem \ref{theo:asymp}.

Now suppose the opposite of  Corollary \ref{cor:Vweakstable}. Then, there exists $\epsilon>0$ and $\phi \in V^\ast$ and a sequence $(t_n), t_n \rightarrow \infty$ such that
\begin{align}
 | \langle  \phi  ,u(t_n) - u^\ast   \rangle |  \geq \varepsilon, \label{eq:bepsilon}
\end{align}
for all $n$. Write $u_n := u(t_n)$. Since $V$ is reflexive and $(u_n)$ is bounded according to \eqref{uVbounded}, we obtain from Alaoglu's compactness theorem that there is a subsequence again denoted by $(u_n)$ which converges weakly in $V$. As $V \subset \subset H$ the subsequence $(u_n)$ converges strongly in $H$, therefore the limit is $u^\ast$ which is in contradiction with \eqref{eq:bepsilon}.
\end{proof}

\section{Concluding remarks \label{sec:conc}}

In this work we have shown well-posedness of a nonlinear reaction-diffusion equation for general core-shell geometry. Furthermore, the corresponding stationary solutions are unique and asymptotically stable in a suitable topology. These results extend the model by \cite{KING19} which only considers stationary solution for spherical core-shell geometry.

The well-posedness theorem, Theorem \ref{theo:wp}, allows us to define a semi-dynamical system: $(D(A), \{S(t)\}_{t\geq 0})$.  This can be used to prove the existence of a global attractor following techniques in \cite{robinson2001infinite}. We expect that Theorem 10.13 from \cite{robinson2001infinite} can be applied and consequently that the global attractor is equal to the unique stationary solution. We note that the techniques from \cite{marion1987attractors,robinson2001infinite} can be used to prove well-posedness by only relying on an $L^2$ bound on $f$ and not on the monotonicity of $f$. 
%An alternative approach to proving stability is by performing a linearization around the stationary solution following theory from \cite{henry2006geometric}. Such an approach would most likely be successful in obtaining a local result but our results hold globally and our approach which makes use of the gradient structure is light in technicalities.

In view of general results on parabolic PDE systems, we expect our evolution problem to be well-posed in H\"{o}lder spaces as well, if $\Gamma$ is smooth enough. However, the proof would be rather technical, and the improvement might not be essential from the point of view of the application. Consequently, it was not considered in this work.

Besides oxygen transport also glucose transport is needed to sustain encapsulated donor cells. In \cite{king2022coupled} a coupled glucose oxygen transport model is proposed based on biological assumptions from \cite{BUCH09,BUCH11}. As in \cite{KING19} stationary solutions are considered and a numerical study is performed to find solutions that have concentrations which are above the donor cell survival threshold. The well-posedness and stability of the corresponding PDE is a topic we would like to explore in future work.

\textbf{Acknowledgments:} During this research Thomas de Jong was also affiliated to Xiamen and University of Groningen. This research was supported by JST CREST grant number JPMJCR2014.

\section*{Appendix}
\appendix

\section{Non-dimensional oxygen concentration \label{app:orig0}}

Denote by $v=v(x,t)$ the non-dimensional concentration then the governing equations become 
\begin{align*}
\frac{dv}{dt}-b \Delta v &= -g(v) \\
v |_S &= c_0 \\
[v] |_\Gamma &= 0 \\
\left[ a \frac{\partial v }{\partial n} \right] \Big|_{\Gamma} &=0 \\
v( x,0)  &= v_0(x), 
\end{align*}
where $g: L^2(\Omega) \rightarrow L^2(\Omega)$ is defined by
\[
[g(u)](x) = \mathbbm{1}_{\Omega_1} \varphi(u(x)),
\]
with $\varphi: \R \rightarrow \R$ given by
\begin{align*}
\varphi(z) = \begin{cases} \frac{z}{z+ \hat{c}} & {\rm if \;} z \geq 0 \\
0 & {\rm else}
\end{cases}
\end{align*}
with $\hat{c}>0$. Let $u = -v+c_0$ and we obtain \eqref{eq:govp}-\eqref{eq:ut0}. Setting $c_1 := c_0 + \hat{c}$ we obtain that $[g(v)](x) = [f(u)](x)$ \eqref{eq:newf}.

\bibliographystyle{alpha}
\bibliography{abbreviated.bib}{}

\end{document}